\RequirePackage{fix-cm}
\documentclass[smallextended]{svjour3}       
\smartqed  
\usepackage{graphicx}
\usepackage{cite}
\usepackage{amsmath,amssymb,amsfonts}
\usepackage{algorithmic}
\usepackage{graphicx}
\usepackage{textcomp}
\usepackage{caption}
\usepackage{float}
\usepackage{xcolor}

\begin{document}

\title{Tracking Controller Design for Satellite Attitude Under Unknown Constant Disturbance Using  Stable Embedding
}

\titlerunning{Tracking Controller Design for Satellite Attitude}        

\author{
Wonshick Ko  \and
 Karmvir Singh Phogat  \and
  Nicolas Petit  \and
   Dong Eui Chang
}


\institute{
Wonshick Ko \at
               Hanon Systems,
              95, Sinilseo-ro, Daedeok-gu, Daejeon, Korea \\
              Tel.: +82-42-350-7540\\
              Fax: +82-42-350-3410\\
              \email{wko1@hanonsystems.com}           
           \and
           Karmvir Singh Phogat \at
              Department of Mathematical Engineering and Information Physics, University of Tokyo, 7 Chome-3-1 Hongo, Bunkyo City, Tokyo 113-8654, Japan \\
              \email{karmvir.p@gmail.com} 
            \and
               Nicolas Petit \at
               CAS, MINES ParisTech, Paris, France\\
               Tel.: +33-1-6469-4863\\
               Fax: +33-1-4051-9165\\
               \email{nicolas.petit@mines-paristech.fr}  
             \and
             Dong Eui Chang \at
             School of Electrical Engineering, Korea Advanced Institute of Science and Technology, 291 Daehak-ro, Daejeon, Korea\\
               Tel.: +82-42-350-7440\\
              Fax: +82-42-350-3410\\
              \email{dechang@kaist.ac.kr}  
}

\date{Received: date / Accepted: date}

\maketitle

\begin{abstract}
We propose a tracking control law for the fully actuated rigid body  system in the presence of any unknown constant disturbance by employing quaternions with the stable embedding technique and Lyapunov stability theory. The stable embedding technique extends the attitude dynamics from the set of unit quaternions to the set of quaternions, which is a Euclidean space, such that the set of unit quaternions is an invariant set of the extended dynamics. Such a stable extension of the system dynamics to a Euclidean space allows us to employ well studied Lyapunov techniques in Euclidean spaces such as LaSalle-Yoshizawa's theorem. A robust tracking control law is proposed for the attitude dynamics subject to  unknown constant disturbance and the convergence properties of the tracking control law is rigorously proven. It is demonstrated with the help of numerical simulations that the proposed control law has a remarkable performance even in some challenging situations.
\end{abstract}
\section{Introduction}
The attitude dynamics and the control of a rigid body encounter the unique challenge that the configuration space of attitudes cannot be globally identified with a Euclidean space  \cite{LeChEu19}. More specifically, the attitude representations such as the three-dimensional special orthogonal group ${\rm SO}(3)$ that is composed of $3\times 3$ orthogonal matrices with the determinant of one \cite{Le12,LeChEu19} or the set of unit quaternions ${\rm S}^3$ encounter the same challenge that they cannot be globally identified with a Euclidean space. Therefore, the controller design and the stability analysis of systems on manifolds require sophisticated differential geometric tools which are often difficult to comprehend for ordinary engineers. An alternative to such cumbersome approaches is to stably embed the system dynamics on manifolds into Euclidean spaces \cite{Ch18IJRNC} and then design controllers in these ambient Euclidean spaces. Moreover, the controller designed on the ambient space has a global representation in contrast to a local chartwise representation \cite{ChJiPe16,Ch18IJRNC,KiPhCh19}.

In this article, we employ the stable embedding technique to design a robust tracking control law for a rigid body system under the influence of unknown constant disturbance. To this end, we first stably embed the system dynamics with the configuration space ${\rm S}^3$ into the set of quaternions $\mathbb H$ which is globally identified with $\mathbb R^4$. The configuration manifold ${\rm S}^3$ becomes a local attractor of the extended dynamics that is defined on $\mathbb H$, and the extended dynamics is identical to the attitude dynamics on the configuration manifold ${\rm S}^3$. Second, a robust tracking control law is then designed for the extended system in the Euclidean space $\mathbb H$ using standard control design tools such as Lyapunov functions and LaSalle-Yoshizawa's theorem. For example, the geometric approaches may be discouraging sometimes as we cannot add and subtract two quaternions in $ {\rm S}^3$ which are perfectly valid operations for quaternions in $\mathbb H$. Therefore, the controller designing in $ \mathbb H$ and its stability analysis will be simplified to a large extent. 

Let us review relevant previous work. A robust global attitude stabilizing, not tracking, control law is proposed using unit quaternion feedback in \cite{JoKeWe95} that is robust with respect to uncertainty in system parameters. A quaternion-based hybrid control law for robust global attitude tracking is proposed in \cite{MaSaTe11} that is robust with respect to the angular velocity measurement bias. This control law is refined further  in \cite{XiSu18} such that a saturated output feedback control law is proposed  for global asymptotic attitude tracking of spacecraft subject to actuator constraints and attitude measurements. However, these existing control laws do not account for the disturbance at the designing stage. On the contrary, by employing the geometric concepts, a hybrid robust exponential attitude tracking control law is designed on ${\rm SO}(3)$ in \cite{Le15} which considers a constant unknown disturbance with a known bound at the designing stage. The technique used in estimating the unknown disturbance for designing the adaptive control law in \cite{Le15} is a guiding principle for the proposed robust control law. A reference shifting technique in combination with a geometric method is employed in \cite{LeChEu19} to achieve semi-global tracking on ${\rm SO}(3)$. These geometric techniques employ sophisticated tools from differential geometry that makes controller designing far more tortuous as compared to the proposed technique. The key contributions of the article are: a) A technique of robust controller design on unit quaternions is presented by stably embedding the unit quaternions to Euclidean spaces, and  b)  A robust tracking controller is designed for the rigid body attitude control system that is subject to an unknown constant disturbance.   In this paper, we deal with the fully actuated rigid body system, and we plan to apply the embedding technique in the future to the case of magnetic actuation which has only two degrees of actuation in a special way; Refer to \cite{LoAs04,XiSu18} for more on the magnetic actuation case.

This paper is organized as follows: A robust tracking controller is designed for the rigid body attitude control system by stable embedding the system dynamics into Euclidean spaces in Sec. \ref{sec:main}. The proposed controller performance is demonstrated with numerical experiments in Sec. \ref{sec:numerical}. The concluding remarks and future scope are presented in Sec. \ref{sec:conclusion}.

\section{Main Results}\label{sec:main}
We briefly review quaternions. Quaternions are represented in the form:  $q = a + b\mathbf i + c\mathbf j + d\mathbf k$, where $a$, $b$, $c$, and $d$ are real numbers, and $\mathbf i$ and $\mathbf j$ and $\mathbf k$ are fundamental quaternion units satisfying $-1 = \mathbf i^2 = \mathbf j^2 = \mathbf k^2$ and $\mathbf i = \mathbf j \mathbf k = - \mathbf k \mathbf j$. The scalar part, $a$, of the quaternion is denoted by $q_s$, and the vector part, $b\mathbf i + c\mathbf j + d\mathbf k$, of the quaternion is denoted by $q_v$. Quaternions with $q_s = 0$ are called pure quaternions. The pure quaternion $b\mathbf i + c\mathbf j + d\mathbf k$ of $q$ is conveniently identified with a vector $(b,c,d)\in \mathbb R^3$ and vice versa. Therefore, with a slight abuse of notation, we denote the vector $(b,c,d) \in \mathbb R^3$ and the corresponding pure quaternion $b\mathbf i + c\mathbf j + d\mathbf k$ as $q_v$; however, the notation is be clearly understood from the context. The operator $[\cdot]_s$ selects the scalar part and $[\cdot]_v$ selects the vector part of a given quaternion, i.e. $[q]_s = q_s$ and $[q]_v= q_v$. The set of all quaternions is denoted $\mathbb H$. The product $qp$ of two quaternions $q = q_s + q_v, p = p_s + p_v\in \mathbb H$ can be compactly expressed as
\[
qp = (q_s p_s - \langle q_v, p_v\rangle) + (q_s p_v + p_s q_v + q_v \times p_v),
\]
where $\langle\, , \rangle$ and $\times$ are the dot product and the cross product on $\mathbb R^3$. Note that the set $\mathbb H$ is a vector space that is endowed with quaternion multiplication. Some papers use the nonstandard symbol $\odot$ or $\otimes$ to denote quaternion multiplication, but we do not adopt the notation here.   The conjugation $q \mapsto q^*$ is defined by $q^* = q_s - q_v$, i.e. $(a + b\mathbf i + c\mathbf j + d\mathbf k)^* = a - b\mathbf i - c\mathbf j - d\mathbf k$, and it satisfies $(qp)^* = p^*q^*$ for $q,p \in \mathbb H$.  If $\Omega$ is a pure quaternion, then $q^*\Omega q$ is also a pure quaternion for all $q\in \mathbb H$. The norm of $q =a + b\mathbf i + c\mathbf j + d\mathbf k \in \mathbb H$ is defined by $|q| = \sqrt{a^2 + b^2 + c^2 + d^2}$. It is easy to show $|q|^2 = q^*q = qq^*$.    Let $\text{S}^3 = \{ q\in \mathbb H \mid |q| = 1\}$ denote the set of unit quaternions, and $\mathbb H_0 $ denote the set $\{ q \in \mathbb H \mid q \neq 0 \}$. 

\subsection{Tracking Control of a Rigid Body}
The equations of rotational motion of a rigid body are given by
\begin{subequations}\label{original:dyn:quat}
\begin{align}
\dot q &= \frac{1}{2}q\Omega \label{original:dyn:quat:1},\\
\dot \Omega & = {\mathbb I}^{-1} ( ({\mathbb I } \Omega) \times \Omega)) +  {\mathbb I}^{-1}\tau,
\end{align} 
\end{subequations}
where $q\in {\rm S}^3$ is the attitude of the rigid body, $\Omega \in \mathbb R^3$ is the body-fixed angular velocity, $\tau \in \mathbb R^3$ is control vector and $\mathbb I \in \mathbb R^{3\times 3}$ is the moment-of-inertia matrix of the rigid body. It is worth noting that  $\Omega \in \mathbb R^3$ is identified with a pure quaternion in \eqref{original:dyn:quat:1}, and therefore, $q\Omega$ in \eqref{original:dyn:quat:1} is the quaternion multiplication.
We now extend the system dynamics \eqref{original:dyn:quat}, using the stable embedding technique \cite{Ch18IJRNC} from ${\rm S}^3 \times \mathbb R^3$ to $\mathbb H \times \mathbb R^3$ as 
\begin{subequations}\label{ext:dyn:quat}
\begin{align}
\dot q &= \frac{1}{2}q\Omega - \alpha (|q|^2-1) q,\\
\dot \Omega & = {\mathbb I}^{-1} ( ({\mathbb I } \Omega) \times \Omega) )+  {\mathbb I}^{-1}\tau
\end{align} 
\end{subequations}
with $\alpha>0$.
\begin{lemma}\label{lemma:pos:inv:H0}
The dynamics \eqref{ext:dyn:quat} reduces to \eqref{original:dyn:quat} on ${\rm S}^3 \times \mathbb R^3$. Moreover,  ${\rm S}^3 \times \mathbb R^3$ is an invariant set of \eqref{ext:dyn:quat}, and a local exponential attractor of \eqref{ext:dyn:quat} with the region of convergence $\mathbb H_0 \times \mathbb R^3$.

\begin{proof}
As we know that $|q| = 1$ for each $q \in \text{S}^3$, the system dynamics \eqref{ext:dyn:quat} reduces to \eqref{original:dyn:quat} on $\text{S}^3 \times \mathbb R^3$. Let us define a Lyapunov-like function 
 \[
\mathbb H \ni q \mapsto V(q) = |q|^2 = qq^* \geq 0 \in \mathbb R
\]
such that $\text{S}^{3} = \{q \in \mathbb H \mid V(q) = 1 \}.$ Along the trajectories of 
 \eqref{ext:dyn:quat}, 
 \begin{align*}
 \frac{d}{dt}V &= \dot q q^* + q \dot q^*  \\
 &= \left ( \frac{1}{2}q\Omega - \alpha (|q|^2-1) q\right ) q^*+ q \left ( \frac{1}{2}q\Omega - \alpha (|q|^2-1) q\right )^*\\
 &=  \frac{1}{2}q\Omega q^*- \alpha (|q|^2-1) qq^*+  \frac{1}{2}q\Omega^* q^*- \alpha (|q|^2-1) qq^*\\
 &= -2\alpha (|q|^2-1)|q|^2\\
 &= -2\alpha (V-1)V
 \end{align*}
which  is a nonlinear first-order ODE in $V$ with $V=1$ being an exponentially stable equilibrium point  with the region of convergence $\{V \in \mathbb R \mid V > 0 \}$. Notice that the set $V=1$ corresponds to ${\rm S}^3$ and the set $V>0$ corresponds to $\mathbb H_0$. Hence, ${\rm S}^3 \times \mathbb R^3$ is an invariant set of \eqref{ext:dyn:quat}, and a local exponential attractor of \eqref{ext:dyn:quat} with $\mathbb H_0 \times \mathbb R^3$ as the region of convergence. 
\end{proof}
\end{lemma}
The control objective is to design a tracking control law for the system dynamics \eqref{ext:dyn:quat} to track a pre-defined reference trajectory. To this end, let us consider a reference trajectory $(q_0(t), \Omega_0(t)) \in \text{S}^3 \times \mathbb R^3$ with $t\geq 0$ that satisfies \eqref{original:dyn:quat:1}. Then the corresponding tracking error is defined by 
\begin{equation}\label{def:tracking:error}
e_q = q_0^*q - 1, \quad e_\Omega = \Omega - \Omega_0,
\end{equation}
which satisfy the following error dynamics:
\begin{subequations}\label{error:dyn:quat}
\begin{align}
\dot e_q &= \frac{1}{2} (e_q \Omega_0 - \Omega_0 e_q)+ \frac{1}{2}(1+e_q)e_\Omega  - \alpha (|q|^2 - 1)(1+e_q), \label{eq:dot}\\ 
\dot e_\Omega &= {\mathbb I}^{-1} ( ({\mathbb I } \Omega) \times \Omega) )+  {\mathbb I}^{-1}\tau - \dot\Omega_0. \label{error:dyn:quat:2}
\end{align} 
\end{subequations}
Note that, in the subsequent discussion, we write $e_q = e_{q,s} + e_{q,v}$, where $e_{q,s} = [e_q]_s$ and $e_{q,v} = [e_q]_v$. It is evident from the tracking error \eqref{def:tracking:error} that $|q|^2 = |q_0^*q|^2 = |1+e_q|^2 = 1 + (e_q + e_q^*) + |e_q|^2$, and simplifies to the following identity
\begin{align}\label{eq:Identity}
 |q|^2 = 1 + 2e_{q,s} + (e_{q,s})^2 + |e_{q,v}|^2.
\end{align}
Before we proceed with the design of a feedback control law for tracking, let us prove the following auxiliary result:
\begin{lemma}\label{lemma:V0:dot}
If $V_0(e_q) = \frac{1}{2}|e_q|^2$, then along the trajectory of  \eqref{error:dyn:quat},
\begin{align}\label{dot:V0:quat}
\frac{d}{dt} V_0(e_q) &= -\alpha (2 + e_{q,s} + (|q|^2 -1) )(e_{q,s})^2 \nonumber \\ &+ \frac{1}{2}\langle e_{q,v}, e_\Omega - 2\alpha (e_{q,s} + |q|^2 - 1)e_{q,v} \rangle.
\end{align}

\begin{proof}
By differentiating $V_0(e_q(t)) = \frac{1}{2}e_q(t)^*e_q(t)$ and substituting the error dynamics \eqref{error:dyn:quat}, we get 
\begin{align}
\frac{d}{dt}V_0 &= \frac{1}{2}(\dot e_q^* e_q + e_q^*\dot e_q ) \nonumber \\
&= [e_q^*\dot e_q]_s \nonumber\\
&= \left [\frac{1}{2}(|e_q|^2\Omega_0 - e_q^*\Omega_0e_q)\right ]_s  \nonumber \\
&\quad+ \left [ \frac{1}{2} (e_q^* + |e_q|^2)e_\Omega -\alpha (|q|^2 -1)(e_q^* + |e_q|^2)\right ]_s  \nonumber\\
&= \frac{1}{2} [e_q^*e_\Omega ]_s -\alpha (|q|^2 -1)e_{q,s} -\alpha (|q|^2 -1) |e_q|^2  \nonumber \\
& = \frac{1}{2}\langle e_{q,v}, e_\Omega\rangle -\alpha (|q|^2 -1) (|e_{q}|^2 + e_{q,s}). \label{dot:V0:quat:int}
\end{align}
Employing the identity \eqref{eq:Identity} translates \eqref{dot:V0:quat:int} to \eqref{dot:V0:quat}. That proves the assertion. 
\end{proof}
\end{lemma}
Let
\begin{equation}\label{def:eta:quat}
\eta =-k_q e_{q,v} + 2\alpha (e_{q,s} + |q|^2 - 1)e_{q,v},
\end{equation}
and 
\begin{align}\label{def:eta:dot:quat}
\dot \eta &= -k_q \dot e_{q,v} + 2\alpha (e_{q,s} + |q|^2 - 1)\dot e_{q,v}  \nonumber \\ &\quad\quad+  2\alpha (\dot e_{q,s} -2\alpha(|q|^2-1)|q|^2)e_{q,v},  
\end{align}
where  $ \dot e_{q,s}$ and $ \dot e_{q,v}$  denote the scalar part and the vector part of the right side of \eqref{eq:dot}, respectively.

\begin{theorem} \label{main:thm:quat}
Let
\begin{equation}\label{def:V:quat}
V_{k_1}(e_q,e_\Omega) =  \frac{1}{2}|e_q|^2 + \frac{1}{4k_1} |e_\Omega -\eta|^2
\end{equation}
with $k_1>0$, where $\eta$ is defined in \eqref{def:eta:quat} with $k_q>0$.
Take any two numbers $c$ and $\epsilon$ such that $0<c<2$ and $0\leq\epsilon<\min\{2-\sqrt{2c},1\}$, and let 
\begin{align}\label{def:S:ep:b:k}
S_{\epsilon,c,k_1}&=  \{ (q,\Omega, e_q,e_\Omega) \in \mathbb H \times \mathbb R^3 \times \mathbb H \times \mathbb R^3 \mid  \nonumber \\ &\qquad\qquad | |q|^2-1| \leq \epsilon, V_{k_1}(e_q,e_\Omega) \leq c \}.
\end{align}
 Then, the feedback 
\begin{equation}\label{control:u:quat}
\tau =  - ({\mathbb I } \Omega) \times \Omega+ \mathbb I ( - k_1e_{q,v}-k_\Omega (e_\Omega - \eta) + \dot \eta + \dot \Omega_0)
\end{equation}
with any  $k_\Omega>0$
exponentially stabilizes the error dynamics \eqref{error:dyn:quat} to zero, where the gain $k_1$ in \eqref{control:u:quat} is the same parameter as that used in \eqref{def:V:quat}. Moreover,  the set $S_{\epsilon,c,k_1}$ is a positively invariant region of convergence for the error dynamics \eqref{error:dyn:quat} in the sense that if a trajectory begins in $S_{\epsilon,c,k_1}$  then it remains there forward in time and the tracking error $(e_q(t),e_\Omega (t))$ converges exponentially to zero as $t$ goes to infinity.
\end{theorem}
\begin{proof}
Let us first show that the set 
\begin{equation}\label{def:M:epsilon}
M_\epsilon = \{ (q, \Omega) \in \mathbb H \times \mathbb R^3\mid  | |q|^2-1| \leq \epsilon \}
\end{equation}
  is positively invariant for \eqref{ext:dyn:quat}. Let $V_{\rm aux}(q) = \frac{1}{4} (|q|^2-1)^2$. Then, along the trajectory of \eqref{ext:dyn:quat}, $\frac{d}{dt}V_{\rm aux} = \frac{1}{2}(|q|^2-1) (q^*\dot q + \dot q^* q) = (|q|^2-1) [q^*\dot q]_s = -\alpha (|q|^2-1)^2|q|^2 \leq -\alpha (1-\epsilon) (|q|^2-1)^2\leq -4\alpha (1-\epsilon) V_{\rm aux}$ since $0\leq\epsilon <1$. It follows that $M_\epsilon = V_{\rm aux}^{-1}([0,\epsilon^2/4])$ is positively invariant for \eqref{ext:dyn:quat}.

By employing Lemma \ref{lemma:V0:dot} with \eqref{error:dyn:quat:2}, \eqref{def:eta:quat} and \eqref{def:eta:dot:quat}, for any $(q,\Omega,e_q,e_\Omega) \in S_{\epsilon,c,k_1}$, we conclude the following: 
\begin{align}
\frac{d}{dt} V_{k_1} &= -\alpha (2 + e_{q,s} + (|q|^2 -1) )(e_{q,s})^2 \nonumber \\ &\quad + \frac{1}{2}\langle e_{q,v}, e_\Omega - 2\alpha (e_{q,s}  + |q|^2 - 1)e_{q,v} \rangle \nonumber \\ &\quad + \frac{1}{2k_1} \langle e_\Omega - \eta, \dot e_\Omega -\dot \eta\rangle \nonumber  \\
&=-\alpha (2 + e_{q,s} + (|q|^2 -1) )(e_{q,s})^2  \nonumber \\ &\quad + \frac{1}{2} \langle e_{q,v}, -k_{v}e_{q,v} + e_\Omega -\eta\rangle  \nonumber \\ &\quad + \frac{1}{2k_1} \langle e_\Omega - \eta, - k_1e_{q,v}-k_\Omega (e_\Omega - \eta)\rangle \nonumber  \\
&\leq -\alpha(2-\sqrt{2c} -\epsilon)(e_{q,s})^2 \! - \frac{k_q}{2}  |e_{q,v}|^2 \! \nonumber \\ &\quad -  \frac{k_\Omega}{2k_1} | e_\Omega - \eta|^2 \label{dot:V:quat}
\end{align}
which follows from the fact that for all $(q,\Omega,e_q,e_\Omega) \in S_{\epsilon,c,k_1}$, we have $| |q|^2 - 1| \leq \epsilon$ and $|e_{q,s}|  \leq \sqrt{2V_{k_1}(e_q,e_\Omega) } \leq \sqrt{2c}$.  The Lyapunov function $V_{k_1}(e_q,e_\Omega)$ is quadratic and positive-definite in $(e_q, e_\Omega -\eta)$, and  the right side of \eqref{dot:V:quat} is quadratic and negative-definite in $(e_q, e_\Omega -\eta)$. Therefore, the set $S_{\epsilon,c,k_1}$ is positively invariant for \eqref{ext:dyn:quat} and \eqref{error:dyn:quat}, and that $(e_q, e_\Omega - \eta)$ converges exponentially to zero as $t$ tends to infinity.  In view of the definition of $\eta$ in \eqref{def:eta:quat}, it is easy to show that the closed-loop error dynamics \eqref{error:dyn:quat} with the control \eqref{control:u:quat} is exponentially stable  on $S_{\epsilon,c,k_1}$. The positive invariance of $S_{\epsilon,c,k_1}$  is now trivial to show. 
\end{proof}

Since  the original dynamics \eqref{original:dyn:quat} are defined on $\text{S}^3 \times \mathbb R^3$, let us evaluate Theorem \ref{main:thm:quat} on $\text{S}^3 \times \mathbb R^3$. When $|q|=1$, we have $ |e_q| = |q_0^*q -1| = |q-q_0| \leq 2$, where the equality holds if and only if $ q = -q_0$.   We now show the almost semi-global property  of the controller  on ${\rm S}^3$ in the following corollary.
\begin{corollary}
 For any initial state $(q, \Omega, e_q,e_\Omega)$ with $|q| = 1$ and $ |e_q| <2$, there are numbers $c$, $\epsilon$  and $k_1$ satisfying $0<c<2$,  $0\leq\epsilon<\min\{2-\sqrt{2c},1\}$ and $k_1>0$ such that the region of convergence $S_{\epsilon, c, k_1}$ defined in \eqref{def:S:ep:b:k} contains the point $(q, \Omega, e_q,e_\Omega)$. 
 
\begin{proof}
This can be easily seen if we take the limit of $S_{\epsilon, c, k_1}$ as $c\rightarrow 2^-$, $\epsilon \rightarrow 0^+$ and $k_1 \rightarrow \infty$. 
\end{proof}
\end{corollary}


\subsection{Robust Tracking Control of a Rigid Body in the Presence of Unknown Constant Disturbance}
The equation of motion of the rigid body system is given by
\begin{subequations}\label{original:dyn:quat:Del}
\begin{align}
\dot q &= \frac{1}{2}q\Omega \label{original:dyn:quat:Del:1},\\
\dot \Omega & = {\mathbb I}^{-1} ( ({\mathbb I } \Omega) \times \Omega)) +  {\mathbb I}^{-1}\tau + {\mathbb I}^{-1} \Delta,
\end{align} 
\end{subequations}
where $(q,\Omega) \in {\rm S}^3 \times \mathbb R^3$ is the state of the system, $\tau \in \mathbb R^3$ is the control, and $\Delta \in \mathbb R^3$ is an unknown constant disturbance. In real applications, disturbances are not constant but slowly  varying in time. Hence, the constance assumption on the disturbance is a good approximation to slowly varying disturbances.

Using stable embedding, as in the previous section, we extend the dynamics \eqref{original:dyn:quat:Del} from $\text{S}^3 \times \mathbb R^3$ to $\mathbb H \times \mathbb R^3$ as 
\begin{subequations}\label{ext:dyn:quat:Del}
\begin{align}
\dot q &= \frac{1}{2}q\Omega - \alpha (|q|^2-1) q \label{ext:dyn:quat:Del1},\\
\dot \Omega & = {\mathbb I}^{-1} ( ({\mathbb I } \Omega) \times \Omega) )+  {\mathbb I}^{-1}\tau + {\mathbb I}^{-1} \Delta
\end{align} 
\end{subequations}
with some $\alpha>0$. Consider a reference trajectory $(q_0(t), \Omega_0(t)) \in {\rm S}^3 \times \mathbb R^3$ with $t\geq 0$ that satisfies \eqref{original:dyn:quat:Del:1} or  \eqref{ext:dyn:quat:Del1}, i.e. $\dot q_0 (t) = (1/2)q_0(t)\Omega_0(t)$ for all $t\geq 0$. Our goal is to design a tracking controller such that the trajectory of \eqref{ext:dyn:quat:Del} asymptotically converges to the reference trajectory as time goes to infinity even in the presence of the disturbance $\Delta$. With the similar computation as carried out for \eqref{error:dyn:quat}, the corresponding tracking error $(e_q,e_\Omega)$, defined in \eqref{def:tracking:error}, for the system dynamics \eqref{ext:dyn:quat:Del} satisfies
\begin{subequations}\label{error:dyn:quat:Del}
\begin{align}
\dot e_q &= \frac{1}{2} (e_q \Omega_0 - \Omega_0 e_q) + \frac{1}{2}(1+e_q)e_\Omega 
- \alpha (|q|^2 - 1)(1+e_q), \label{eq:dot:Del}\\ 
\dot e_\Omega &= {\mathbb I}^{-1} ( ({\mathbb I } \Omega) \times \Omega) )+  {\mathbb I}^{-1}\tau +  {\mathbb I}^{-1}\Delta - \dot\Omega_0. \label{error:dyn:quat:Del:2}
\end{align} 
\end{subequations}
The control objective is to design a control law that asymptotically stabilizes the error dynamics \eqref{error:dyn:quat:Del} to zero. We propose the following form of control law:
\begin{equation}\label{control:u:quat:Delta}
\tau =  - ({\mathbb I } \Omega) \times \Omega+ \mathbb I ( - k_1e_{q,v}-k_\Omega (e_\Omega - \eta) + \dot \eta + \dot \Omega_0 ) - \bar \Delta
\end{equation}
with $k_1>0$  and $k_\Omega >0$, where $\eta$ and $\dot \eta$ are  defined in \eqref{def:eta:quat} and \eqref{def:eta:dot:quat}, and  
 $\bar \Delta $ is generated by
\begin{equation}\label{delta:dot}
\dot{\bar \Delta} = \frac{k_\Delta}{2k_1}\mathbb I^{-1}(e_\Omega - \eta)
\end{equation}
with initial condition $\bar \Delta (0) = 0$ and $k_\Delta >0$. Moreover, considering the fact that the disturbance $\Delta$ is constant, the disturbance error $e_\Delta = \Delta - \bar\Delta$ satisfies
\begin{equation}\label{dot:bar:eDelta}
\dot e_\Delta = -\frac{k_\Delta}{2k_1}\mathbb I^{-1}(e_\Omega - \eta).
\end{equation}
We shall now establish that the proposed control law \eqref{control:u:quat:Delta} enables the system dynamics \eqref{ext:dyn:quat:Del} to track the reference asymptotically. In other words, the proposed control law \eqref{control:u:quat:Delta} asymptotically stabilizes the tracking error dynamics \eqref{error:dyn:quat:Del} to zero.
\begin{theorem} \label{main:thm:quat:Del} 
Assume that there is a known number $\delta>0$ such that $\| \Delta \| \leq \delta$.
Let
\begin{equation}\label{def:V:quat:Delta}
V(e_q,e_\Omega ,e_\Delta) = V_{k_1}(e_q,e_\Omega) + \frac{1}{2k_\Delta}|e_\Delta |^2,
\end{equation}
where the function $ V_{k_1}$ is defined in \eqref{def:V:quat}. Then, when the feedback \eqref{control:u:quat:Delta} together with \eqref{delta:dot} is applied to \eqref{error:dyn:quat:Del},  the tracking error trajectory $(e_q(t), e_\Omega (t))$ asymptotically converges to zero as $t$ goes to infinity, for any initial state $q(0) \in \mathbb H$  and any initial tracking error $(e_q(0), e_\Omega(0))$ that satisfy
\begin{equation}\label{qzero:less}
  ||q(0)|^2 -1| \leq \epsilon
\end{equation}
and
\begin{equation}\label{Vk1:zero}
V_{k_1} (e_q(0), e_\Omega (0)) \leq c - \frac{\delta^2}{2 k_\Delta},
\end{equation}
where the constants $c$, $\epsilon$ and $k_\Delta$ are such that $0<c<2$, $0\leq\epsilon<\min\{2-\sqrt{2c},1\}$ and $k_\Delta > \delta^2/2c$. 

\begin{proof}
First, consider the fact that  the set $M_\epsilon$ defined in \eqref{def:M:epsilon}
  is positively invariant for \eqref{ext:dyn:quat:Del} as established in the proof of Theorem \ref{main:thm:quat}.  We now show that the set 
  \begin{align}\label{def:S:Delta}
S=  &\{ (q,\Omega, e_q,e_\Omega, e_\Delta) \in \mathbb H \times \mathbb R^3 \times \mathbb H \times \mathbb R^3 \times \mathbb R^3 \mid  \nonumber \\
&\quad  | |q|^2-1| \leq \epsilon, V(e_q,e_\Omega, e_\Delta) \leq c \}
\end{align}
is positively invariant. Note that for any $(q,\Omega,e_q,e_\Omega, e_\Delta) \in S$, we have  
\[
| |q|^2 - 1| \leq \epsilon \quad \text{and} \quad |e_{q,s}|  \leq \sqrt{2V(e_q,e_\Omega, e_\Delta) } \leq \sqrt{2c}.
\]
  As computed in \eqref{dot:V:quat}, with the help of \eqref{dot:bar:eDelta}, it is straightforward to show that along the closed-loop trajectory with any initial condition in $S$,
\begin{align*}
\dot V &\leq  -\alpha(2-\sqrt{2c} -\epsilon)(e_{q,s})^2 - \frac{k_q}{2}  |e_{q,v}|^2 \\ &\quad-  \frac{k_\Omega}{2k_1} | e_\Omega - \eta|^2 + \frac{1}{2k_1}\langle e_\Omega - \eta, \mathbb I^{-1}e_\Delta \rangle + \frac{1}{k_\Delta}\langle e_\Delta, \dot e_\Delta\rangle\\
&=  -\alpha(2-\sqrt{2c} -\epsilon)(e_{q,s})^2 - \frac{k_q}{2}  |e_{q,v}|^2 -  \frac{k_\Omega}{2k_1} | e_\Omega - \eta|^2.
\end{align*}  
By LaSalle-Yoshizawa's theorem or Theorem 8.4 in \cite{Kh02}, the trajectory asymptotically converges to the set 
\[
\alpha(2-\sqrt{2c} -\epsilon)(e_{q,s})^2 + \frac{k_q}{2}  |e_{q,v}|^2+  \frac{k_\Omega}{2k_1} | e_\Omega - \eta|^2 = 0, 
\]
which only consists of the set $\{ e_q = 0, e_\Omega = 0\}$. Hence, $\|(e_q(t),e_\Omega (t))\| \rightarrow 0$ as $t\rightarrow \infty$.  In other words, the set $S$ is positively invariant, and any trajectory starting in $S$ satisfies $\lim_{t\rightarrow \infty}\| (e_q(t), e_\Omega (t))\| = 0$. In particular, if any initial state satisfies \eqref{qzero:less} and \eqref{Vk1:zero}, then it belongs to $S$, i.e.,
 \begin{align*}
 V|_{t=0} = V_{k1}(e_q(0),e_\Omega (0)) + \frac{1}{2k_\Delta}|e_\Delta(0)|^2 \leq c
 \end{align*}
 and
 \[
 ||q(0)|^2 -1| \leq \epsilon,
 \]
 where $|e_\Delta(0)|  = |\Delta - \bar \Delta (0)| = |\Delta |\leq \delta $ is used. Therefore, the tracking error trajectory $(e_q(t), e_\Omega (t))$ asymptotically converges to zero as $t$ tends to infinity. This proves the assertion.
\end{proof}
\end{theorem}

The following corollary is about almost semi-global property on the configuration space ${\rm S}^3$ of the tracking control law. 
\begin{corollary}
Suppose that an upper bound $\delta>0$ for the unknown constant disturbance $\Delta$ is known. Then, 
 for any initial state  with $|q(0)| = 1$ and $ |e_q(0)| <2$, there are numbers $c$,  $k_1$ and $k_\Delta$ satisfying $0<c<2$,   $k_\Delta > \delta^2/2c$ and $k_1>0$ such that the initial state satisfies \eqref{Vk1:zero} as well as \eqref{qzero:less}. 
 
\begin{proof}
Since $|e_q(0)|^2/2 <2$, we can take  $c$ such that $|e_q(0)|^2/2 < c<2$. Then, take $k_1$ and $k_\Delta$ large enough such that 
\[
 \frac{1}{4k_1}|e_\Omega (0) - \eta(0)|^2 + \frac{\delta^2}{2k_\Delta} <c- \frac{1}{2}|e_q(0)|^2.
 \]
 Then, \eqref{Vk1:zero} is satisfied. 
 
\end{proof}
\end{corollary}

\begin{remark}
If we assume that $\ddot \Omega_0(t)$, $0\leq t<\infty$, is bounded in Theorem \ref{main:thm:quat:Del}, then we can use Barbalat's lemma (Lemma 8.2 in \cite{Kh02}) to prove $\lim_{t\rightarrow \infty}| e_\Delta (t)| = 0$.   The proof goes as follows. Plug \eqref{control:u:quat:Delta} into \eqref{error:dyn:quat:Del:2} to obtain 
\begin{equation}\label{dot_e_Omega_2}
\dot e_\Omega =  - k_1e_{q,v}-k_\Omega (e_\Omega - \eta) + \dot \eta   +  {\mathbb I}^{-1}e_\Delta.
\end{equation}
From \eqref{eq:dot:Del} and \eqref{dot_e_Omega_2}, one can see that $\ddot e_\Omega (t)$ is bounded in time if $\ddot \Omega_0(t)$, $0\leq t<\infty$, is bounded, which implies that $\dot e_\Omega (t)$ is uniformly continuous. We already know that its integral $e_\Omega (t)$ satisfies $\lim_{t\rightarrow\infty}e_\Omega(t) = 0$. Hence, by Barbalat's lemma, we have $\lim_{t\rightarrow\infty}\dot e_\Omega(t) = 0$. Since $\lim_{t\rightarrow\infty} e_q(t) = 0$, it follows from \eqref{dot_e_Omega_2} that $\lim_{t\rightarrow\infty} e_\Delta(t) = 0$.

\end{remark}

\section{Simulation Results}\label{sec:numerical}
We carry out three simulation studies to demonstrate the performance of the robust tracking control law \eqref{control:u:quat:Delta}. The following data has been considered for the numerical simulations:
\begin{enumerate}
\item The moment of inertia matrix:
\[
\mathbb{I}=\text{diag}(4.250,4.337,3.664).
\]
\item The control gains in  \eqref{control:u:quat:Delta} are given by $k_1=3$ and $k_{\Omega}=3$.  
\item The gain $k_\Delta$ in \eqref{delta:dot} will be set to a couple of different values later. 
\item The parameter $\alpha = 1$ in \eqref{original:dyn:quat:Del}.
\end{enumerate}
 The following reference trajectory is chosen for attitude:
\[
q_0(t) = \cos t + \cos t\sin t\textbf{i} + \sin^2 t\textbf{j} + 0\textbf{k} \in {\rm S}^3,
\]
which induces the following reference trajectory for angular velocity and acceleration:
\[
\Omega_0(t) = (2\cos^3 t,\, (2+2\cos^2 t)\sin t,\, -2 \sin^2 t)
\]
and
\[
\dot{\Omega}_0(t) = (-6\cos^2 t \sin t,\, (-2+6\cos^2 t)\cos t,\, -4\sin t \cos t).
\]
The initial condition for the system \eqref{original:dyn:quat:Del} is given by
\[ 
q(0) = -q_0(0) \in {\rm S}^3, \quad \Omega(0) = \Omega_0({\pi}/{6}) \in \mathbb R^3
\]
in terms of the reference trajectory $(q_0(t), \Omega_0(t))$. Notice that the initial attitude tracking error has magnitude $|q(0) - q_0(0)| = 2$, which is the maximum possible attitude tracking error on unit quaternions, so this tracking error does not satisfy \eqref{Vk1:zero} since $c<2$. Hence, this is a  challenging initial condition for tracking.  

In the first case study, we consider the  constant disturbance vector and the gain
\[
\Delta = (1,1,1), \quad  k_\Delta = 1000
\]
for \eqref{delta:dot}, and the simulation results are reported in Fig.\ref{fig:constant:dis}. It is observed in the first two subfigures of Fig.\ref{fig:constant:dis} that both the attitude tracking error and the angular velocity tracking error converge to zero as time goes to infinity.  It can be seen in the third subfigure of Fig.\ref{fig:constant:dis} that the disturbance estimate $\bar \Delta (t)$ also converges to the true value of disturbance $\Delta$.  

\begin{figure}[htb]
 \centering
        \includegraphics[height=1.5in]{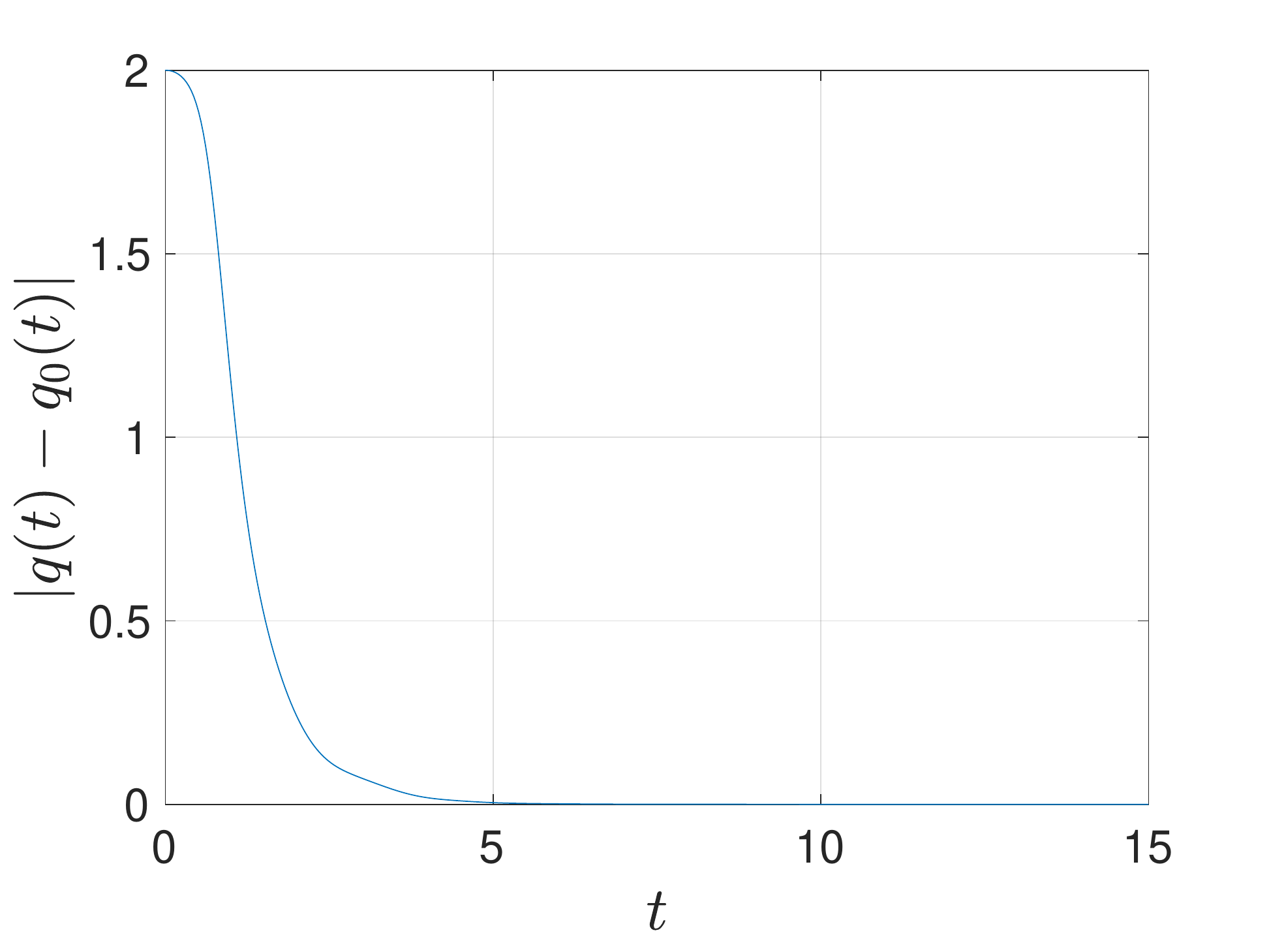}
        \includegraphics[height=1.5in]{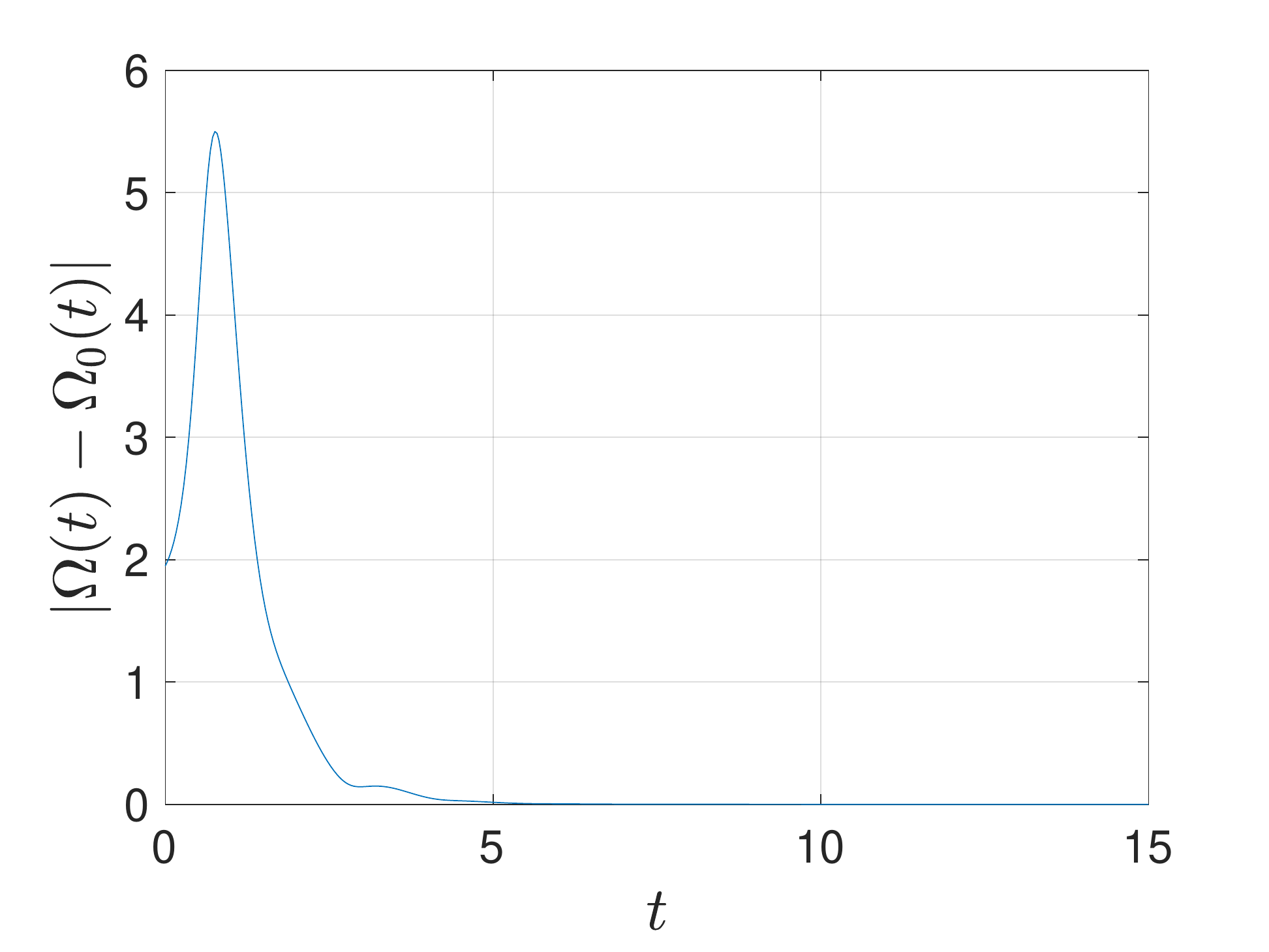}
        \includegraphics[height=1.5in]{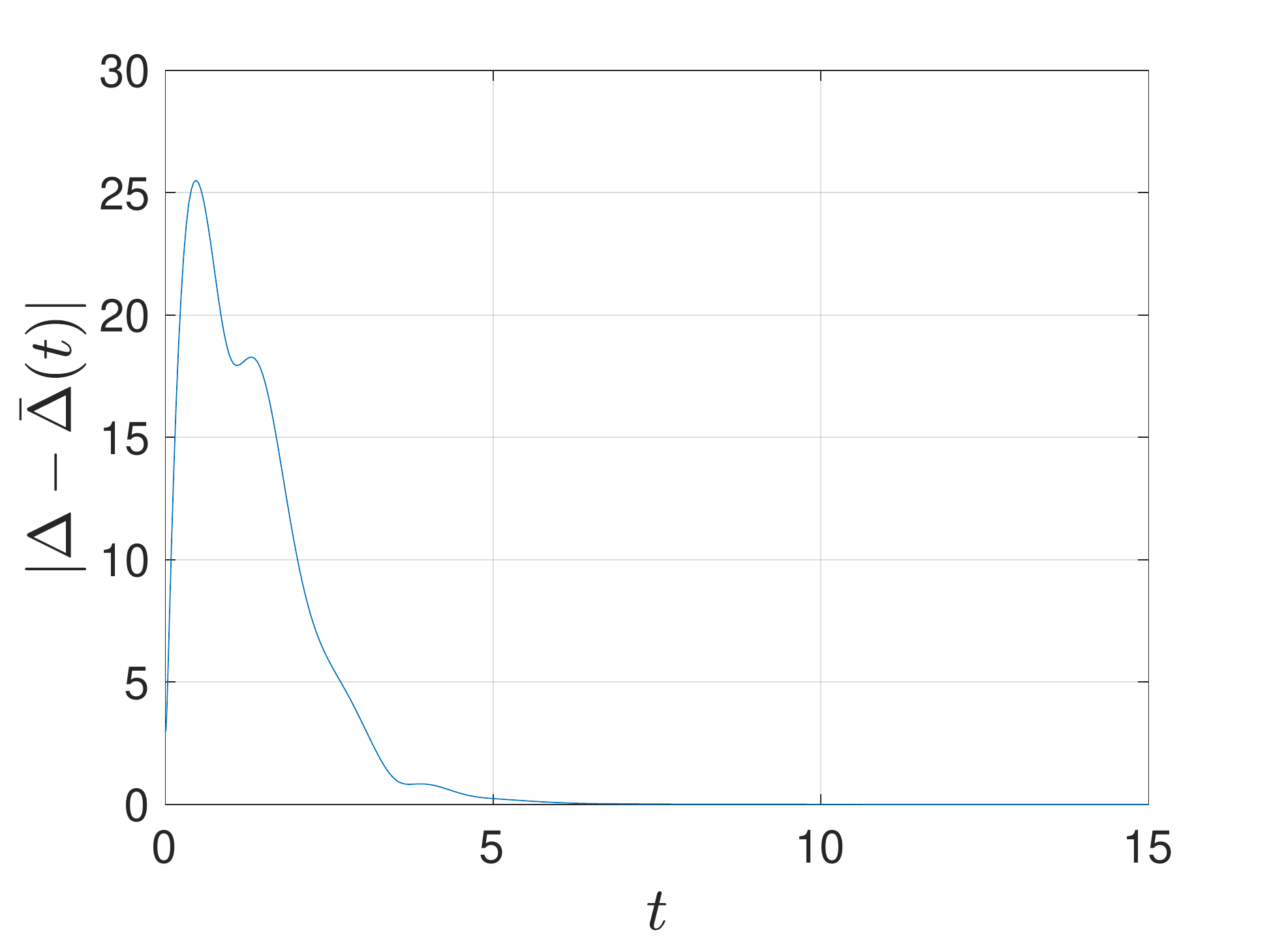}
    \caption{\label{fig:constant:dis}Simulation result for the constant disturbance $\Delta = (1,1,1)$; (a) Attitude tracking error, (b) Velocity tracking error, (c) Disturbance tracking error.}
 \end{figure}

In the second case study, we consider the time-varying disturbance vector
\begin{equation}\label{sinu:dis}
\Delta = \cos(0.5 t) \times (1,1,1),
\end{equation}
which violates our assumption that disturbance is constant. We choose the gain  $k_\Delta = 1000$ for \eqref{delta:dot} and report the results in Fig.\ref{fig:sinu:dis}. The tracking error does not go to zero and that is not a surprise since the control law \eqref{control:u:quat:Delta} and \eqref{delta:dot} is designed for constant disturbances. However, the tracking performance shown in Fig.\ref{fig:sinu:dis} is acceptable for all practical purposes. 
\begin{figure}[htb]
\centering
         \includegraphics[height=1.5in]{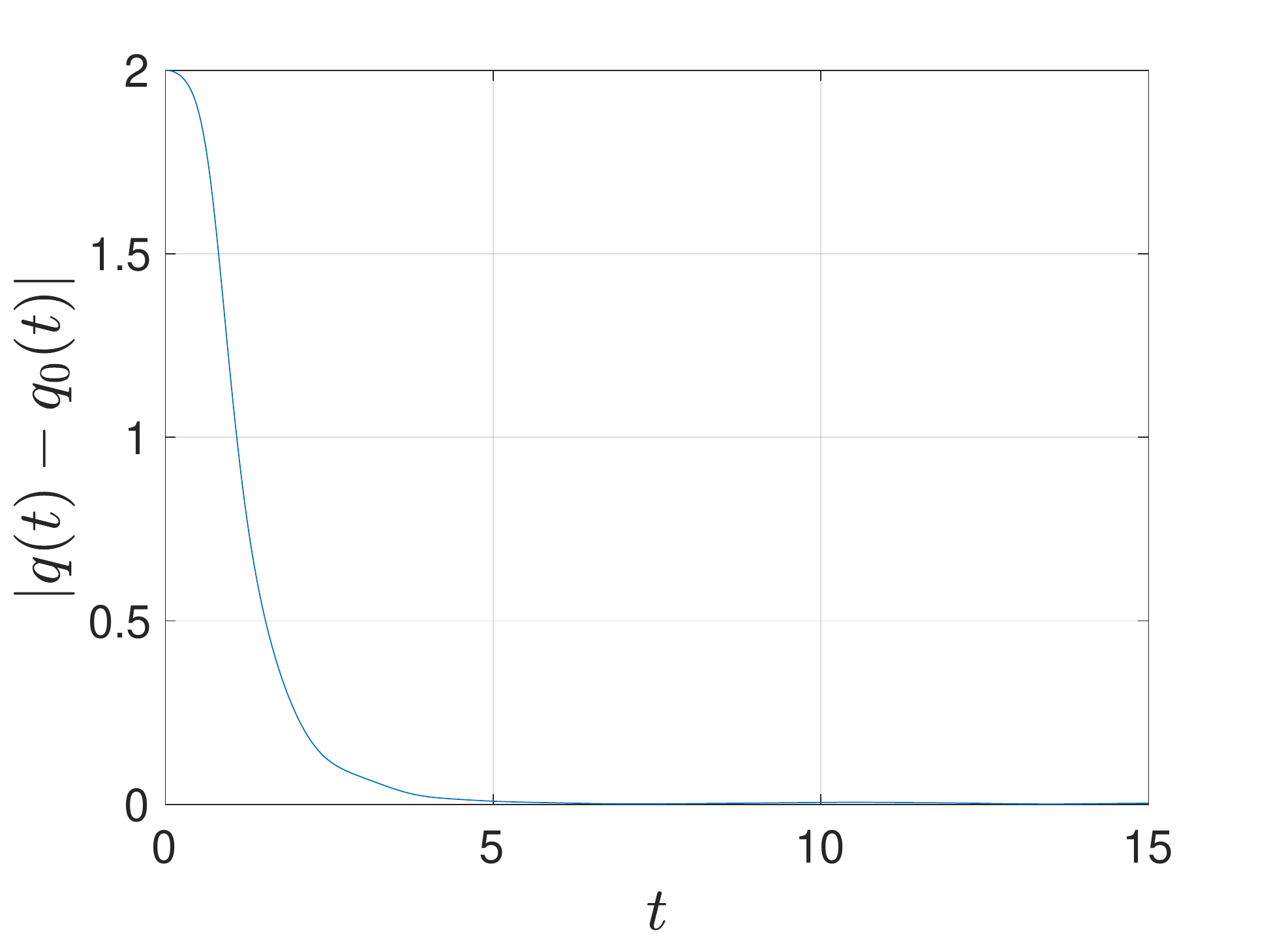}
          \includegraphics[height=1.5in]{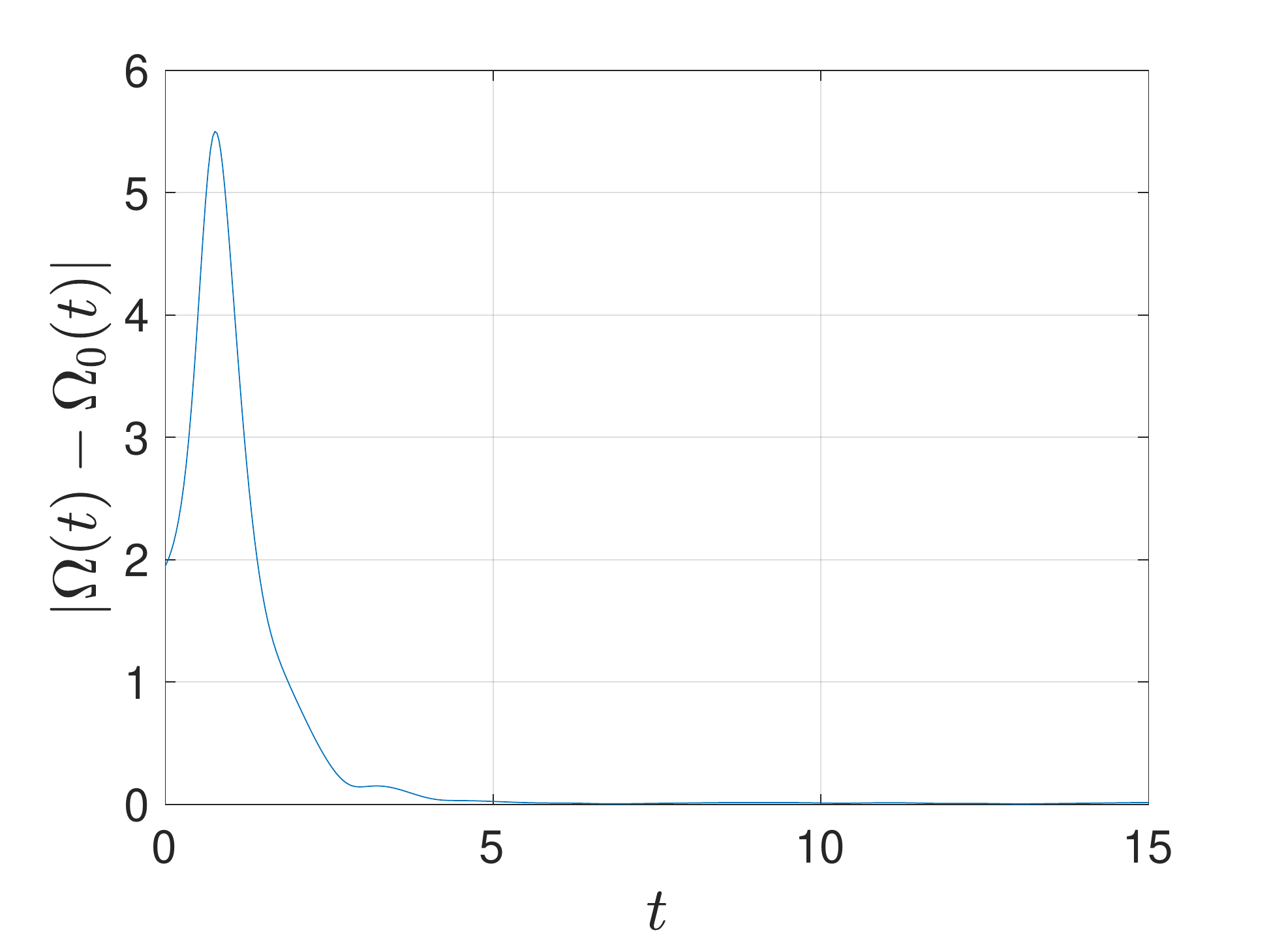}
         \includegraphics[height=1.5in]{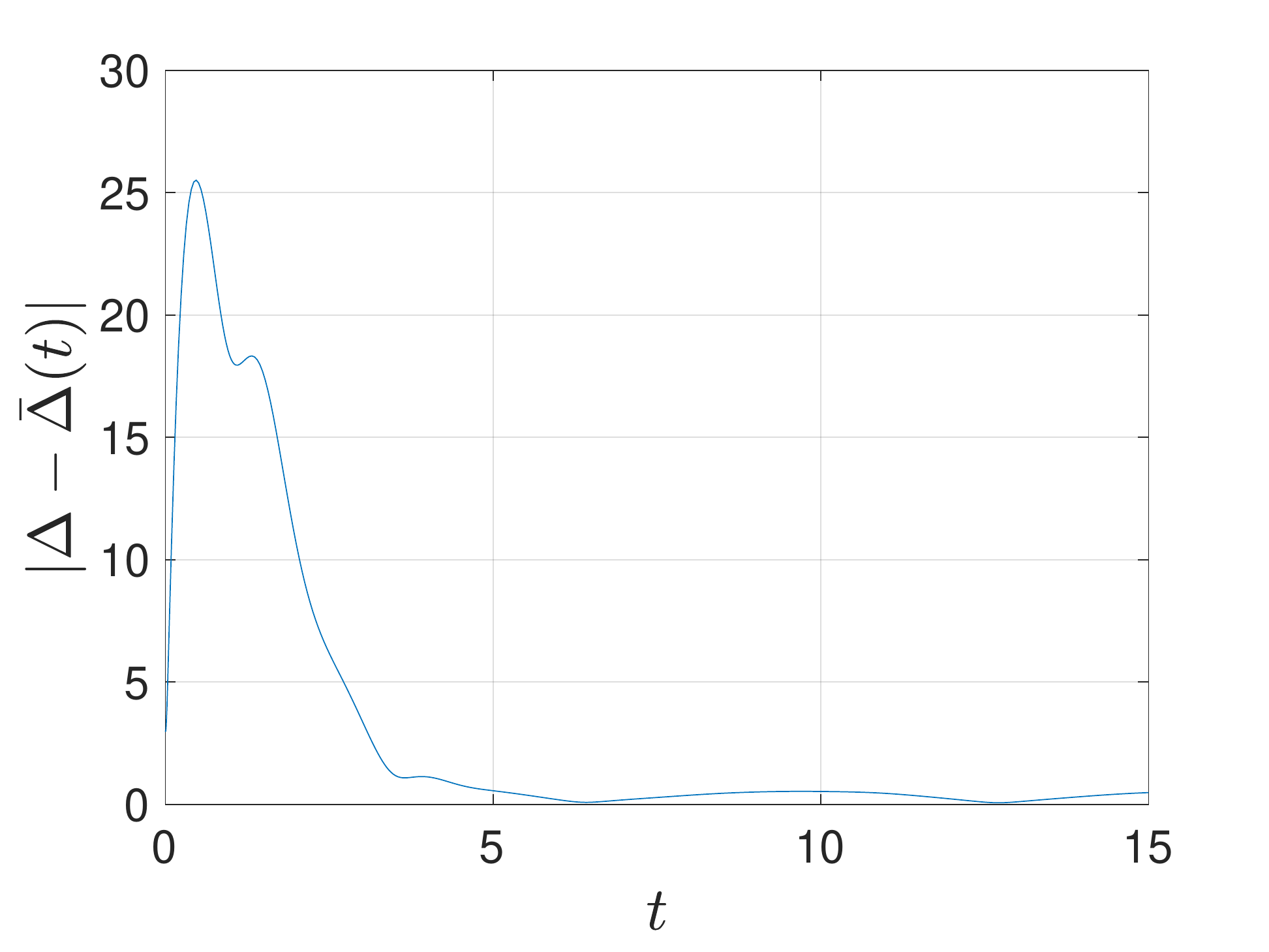}

      \caption{\label{fig:sinu:dis}Simulation result for the constant disturbance $\Delta = \cos(0.5t)\times(1,1,1)$ ; (a) Attitude tracking error, (b) Velocity tracking error, (c) Disturbance tracking error.}
 \end{figure}

In the third case study, we carry out simulation with the non-robust control law \eqref{control:u:quat} while keeping other settings unchanged for the disturbance \eqref{sinu:dis} and compare its performance with the robust control law. The velocity tracking errors for both control law are shown in Fig.\ref{fig:comparison:vel}. It is evident from Fig.\ref{fig:comparison:vel} that the tracking performance of the robust tracking control law is much better than that of the non-robust control law. 

 \begin{figure}[htb]
 \centering
 \includegraphics[height=1.8in]{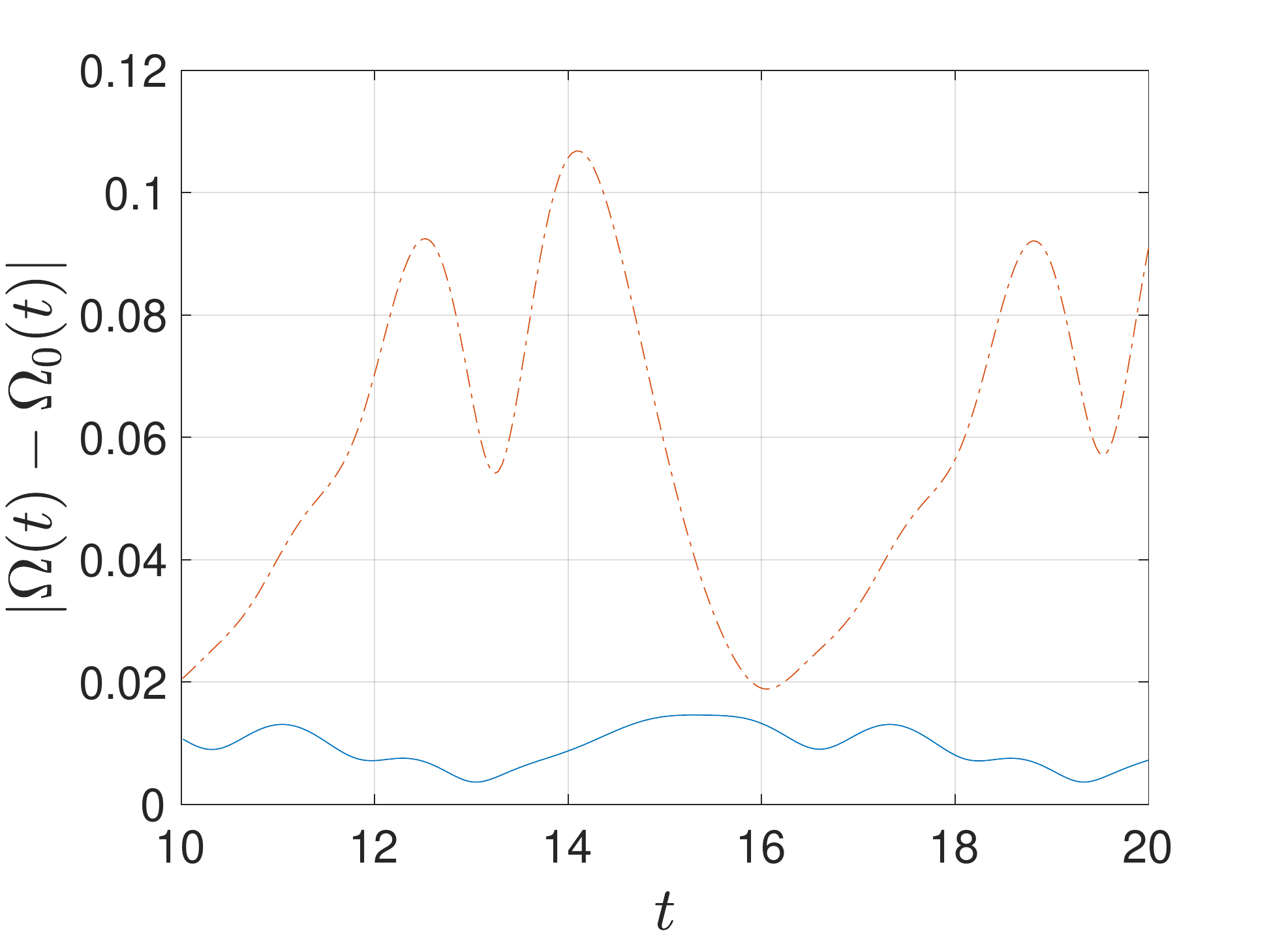}
 \caption{\label{fig:comparison:vel}Velocity tracking errors in presence of a sinusoidal disturbance with the robust tracking control law \eqref{control:u:quat:Delta} and \eqref{delta:dot} (blue solid) and the non-robust tracking control law \eqref{sinu:dis} (red dash-dot).}
 \end{figure}
 %
 
 \section{Conclusion}\label{sec:conclusion}
In this article, we proposed robust tracking control law for the attitude dynamics. The attitude dynamics is represented in quaternions that is stably embedded in Euclidean space and the tracking controller is then designed in that Euclidean space. The proposed control law enables asymptotic tracking of the reference by the system dynamics. In future, the proposed technique will be developed for more general class of systems and disturbances.   In particular, we would like to generalize the proposed technique for angular velocity and torque estimation of the rigid body \cite{MaPe17}.

\end{document}